\documentclass[11pt]{article}

\usepackage[margin=1.15in]{geometry}
\usepackage{amsmath,amssymb,amsthm,mathtools}
\usepackage{enumitem}
\usepackage{microtype}
\usepackage[hidelinks]{hyperref}

\newtheorem{theorem}{Theorem}[section]
\newtheorem{proposition}[theorem]{Proposition}
\newtheorem{remark}[theorem]{Remark}
\newtheorem{corollary}[theorem]{Corollary}
\newtheorem{lemma}[theorem]{Lemma}

\newcommand{\M}{\Sigma}
\newcommand{\R}{\mathbb{R}}
\newcommand{\E}{\mathbb{E}}

\newcommand{\trace}{\operatorname{tr}}
\newcommand{\tr}{\operatorname{tr}}
\newcommand{\Cov}{\operatorname{Cov}}
\newcommand{\Var}{\operatorname{Var}}

\newcommand{\covariance}{\operatorname{Cov}}
\newcommand{\Id}{\operatorname{Id}}
\newcommand{\diag}{\operatorname{diag}}

\newcommand{\abs}[1]{\left\vert#1\right\vert}

\newcommand{\brac}[1]{\left(#1\right)}
\newcommand{\scalar}[1]{\left \langle #1 \right \rangle}

\newcommand{\eps}{\varepsilon}

\numberwithin{equation}{section}

\begin{document}

\title{Thin-Shell implies small-ball deviation via Gaussian tilts}
\date{}

\author{Silouanos Brazitikos\textsuperscript{1} and Emanuel Milman\textsuperscript{2}}

\footnotetext[1]{Department of Mathematics and Applied Mathematics, University of Crete, 70013 Heraklion, Crete, Greece. Email: silouanb@uoc.gr.} 
\footnotetext[2]{Department of Mathematics, Technion -- Israel Institute of Technology, Haifa, Israel. Email: emilman@tx.technion.ac.il. \\
The research leading to these results is part of a project that has received funding from the European Research Council (ERC) under the European Union's Horizon 2020 research and innovation programme (grant agreement No 101001677).}

\begingroup    \renewcommand{\thefootnote}{}        \endgroup

\maketitle

\begin{abstract}
We show that uniform thin-shell estimates for isotropic log-concave measures $\mu$ on $\R^n$ yield precise and explicit deviation estimates for the Euclidean norm $|X|$ below the expectation, improving the square-root dependence of Klartag--Lehec to a quadratic one (which is best possible, up to numeric constants). 
Using the recent Chen--Klartag sharp variance bound, we deduce:
\[
 \mu\left(|X|\le \sqrt{n} - s \right ) \le \exp \left \{  - \frac{s^2}{4} \left (1 + O\brac{\frac{s}{\sqrt{n}}} \right ) \right \} \;\;\; \forall s \in (0,\sqrt{n}) . 
\]
In particular, this yields a new and transparent proof that Thin-Shell implies Slicing (by passing through small-ball estimates). Our method is based on using central Gaussian tilts, recently introduced by Brazitikos, which may be thought of as a deterministic version of Eldan's stochastic localization. An anisotropic variant (when $\mu$ has general covariance structure) of these deviation estimates is also obtained. 
\end{abstract}

\section{Introduction}

Let $\mu$ be a log-concave probability measure on $\R^n$ with density $f$ (which can and will be assumed continuous on its support). Recall that an absolutely continuous measure $\mu$ on $\R^n$ is called log-concave if $\log f : \R^n \rightarrow \R \cup \{-\infty\}$ is concave. 
Let $X\sim\mu$, and define
\[
\Cov(\mu) = \E (X - \E X)\otimes (X - \E X) . 
\]
The measure $\mu$ is called isotropic if
\[
    \E X=0,\qquad \Cov(\mu)=\Id . 
\]
The isotropic constant of $\mu$ is defined as the following affine-invariant parameter:
\[
L_\mu := \|f\|_\infty^{1/n} \det\Cov(\mu)^{1/(2n)} . 
\]

\smallskip
Define
\[
L_n := \sup  L_\mu ~,~ \bar \sigma_n^2 :=   \sup \frac{1}{n} \Var_{X \sim \mu}(|X|^2) ,
\]
where in both cases the supremum is over all isotropic log-concave measures $\mu$ on $\R^n$. Sometimes a different parameter $\sigma^2_n := \sup \E (|X| - \sqrt{n})^2$ is used in the literature, but it is well-known that it is equivalent to $\bar \sigma_n^2$ up to universal constants (e.g.~\cite[Lemma 1.4]{EldanKlartag}). 

\smallskip
Bourgain's Slicing problem \cite{Bourgain1986,MilmanPajor1989} asks whether $L_n \leq C$ uniformly for all $n \geq 1$; the Thin-Shell conjecture \cite{BobkovKoldobsky,AnttilaBallPerissinaki2003} asks whether $\bar \sigma_n \leq C$ uniformly for all $n \geq 1$. Here and elsewhere, $c ,C, C', C'', \ldots$ denote universal constants, independent of dimension, whose value may change from one occurrence to the next. 
Recently, both the Slicing problem and Thin-Shell conjecture were fully resolved by Klartag and Lehec in two groundbreaking works \cite{KlartagLehec-Slicing,KlartagLehec-ThinShell}, both relying on Eldan's method of stochastic localization \cite{Eldan2013}; we refer to these works for further context and history of these longstanding problems. 

\smallskip
It is easy to check that the Slicing problem is equivalent to a small-ball deviation estimate of the following form (see Proposition \ref{prop:equiv} below):
\begin{equation} \label{eq:intro-SB}
\mu \left ( |X| \le \varepsilon \sqrt{n} \right ) \le (C \varepsilon)^n \;\;\; \forall \varepsilon \in (0,1/C) .
\end{equation}
In fact, it was realized by Dafnis and Paouris \cite{DafnisPaouris2010} that it is enough to replace the exponent $n$ by $c n$ on the right-hand side (see also Bizeul \cite[Section 3.2]{Bizeul2025} and Brazitikos \cite{Brazitikos2026} for simplified arguments). 
In contrast with the case when $\eps \in (0,1/C)$, for fixed $\eps \in (1/C , 1)$, all previous known bounds in (\ref{eq:intro-SB}) have been of the form $C \exp(-c_{1-\eps} n^\alpha)$ for some sub-optimal $\alpha < 1$, such as $\alpha = 1/2$ in \cite{GuedonMilman} or $\alpha=1/4$ in \cite{KlartagLehec-ThinShell}. 

\smallskip
For example, by using the reverse H\"older inequalities for moments of the quadratic polynomial $|X|^2 - n$ (see \cite[Corollary 1.3]{KlartagLehec-ThinShell}), Klartag--Lehec deduced from their universal thin-shell estimate the following (equivalent) deviation estimate:
\begin{equation} \label{eq:intro-KL}
\mu  \left ( \abs{|X| - \sqrt{n}} \ge s \right ) \le C \exp( - c \sqrt{s} )  \;\;\; \forall s > 0 . 
\end{equation}
The square-root dependence on $s$ is known to not be optimal, both in the large-deviation regime above the expectation (owing to Paouris' large-deviation estimates \cite{Paouris2006}), as well as in the small ball-regime below the expectation (as witnessed by setting $s = (1-\varepsilon) \sqrt{n}$, yielding $C \exp(-c_{1-\varepsilon} n^{1/4})$ in (\ref{eq:intro-SB})). 

\smallskip
However, the above deviation estimates already suggest some connection between the Thin-Shell and Slicing problems. Indeed, by introducing a certain Riemannian metric on the space of exponential tilts of a given log-concave measure $\mu$, it was shown by Eldan and Klartag \cite{EldanKlartag} that a positive resolution of the Thin-Shell conjecture would positively resolve the Slicing problem, with the following precise linear dependence:
\begin{equation} \label{eq:intro-Lnsigman}
L_n \leq C \bar \sigma_n . 
\end{equation}

\subsection{Deterministic Proofs}

Very recently, two ``deterministic" proofs (not relying on stochastic methods) of the Slicing and Thin-Shell conjectures have been obtained. By revisiting the moment-map method, Chen and Klartag \cite{ChenKlartag2026} obtained the following remarkable sharp bound, verifying that the thin-shell parameter is maximized for products of exponential distributions:
\[
\bar \sigma_n^2 \leq 8 . 
\]
In addition, Brazitikos \cite{Brazitikos2026} has obtained a proof of the Slicing problem using a novel approach involving Gaussian tilts, relying only on standard tools in asymptotic geometric analysis, such as the Bourgain--Milman reverse Blaschke--Santal\'o inequality, reverse H\"older inequalities between moments of linear functions, etc.~\cite{AsymptoticGeometricAnalysis-Book-I}, coupled with Klartag's refined Lichnerowicz estimate \cite{Klartag2023}. A rudimentary Gaussian tilt approach was also employed by E.~Milman in \cite{Milman2013}. 

\subsection{Our results}

The starting point of our work was the observation that the Gaussian tilt method from \cite{Brazitikos2026} (implemented there just for even measures) can be modified to give a new proof that Thin-Shell implies Slicing. As the above discussion suggests, instead of working with the isotropic constant $L_\mu$ directly, our aim will be to derive a single-radius small-ball estimate, which 
allows us a very direct comparison with the Thin-Shell postulate. 
Consequently, we show that the Chen-Klartag explicit sharp bound on the thin-shell parameter translates into the following \textbf{explicit} small-ball estimate. Note the absence of any unspecified universal constants in its formulation. 

\begin{theorem} \label{thm:intro-main}
Let $\mu$ be an isotropic log-concave probability measure on $\R^n$, and let $X$ be a random vector distributed according to $\mu$. Then for all $\eps\in(0,1)$,
\begin{equation}
\label{eq:intro-lower-tail}
    \mu\left(|X|\le \eps \sqrt n\right)
    \le
    \exp\left\{
        -\frac{n}{8}
            \brac{1-\eps^2
            +\eps^2\log \eps^2}
    \right\} , \\
\end{equation}
or equivalently, setting $s = (1-\eps) \sqrt{n}$, for all $s \in (0 , \sqrt{n})$,
\begin{equation} \label{eq:intro-asymptotics}
    \mu\left(|X|\le \sqrt{n} - s \right ) \le \exp \left \{  - \frac{s^2}{4} \left (1 + O\brac{\frac{s}{\sqrt{n}}} \right ) \right \} . 
\end{equation}
\end{theorem}

This deviation estimate captures the thin-shell variance without any loss, improves the Klartag--Lehec deviation estimate (\ref{eq:intro-KL}) below the expectation from square-root to quadratic in $s$, and captures the correct dimension dependence in the exponent $\exp(- c_{1-\eps} n)$ for all $\eps \in (0,1)$, not just for $\eps \in (0,1/C)$ as in (\ref{eq:intro-SB}). In fact, for each $\eta > 0$, (\ref{eq:intro-lower-tail}) is best possible in the range $\eps \in [\eta , 1]$  (up to numeric constants in the exponent depending on $\eta$), as may be verified for the standard Gaussian probability measure. 
Moreover, since the Chen--Klartag thin-shell estimate $\frac{1}{n} \Var |X|^2 \leq 8$ is sharp for $X = (X_i)_{i=1,\ldots,n}$ where $X_i = E_i -1$ and $E_i$ are i.i.d.~standard exponential random variables \cite{ChenKlartag2026}, it is not hard to check that the constant $\frac{1}{4}$ in (\ref{eq:intro-asymptotics}) is best-possible when $s = o(\sqrt{n})$. 
Applying (\ref{eq:intro-lower-tail}) for $\eps = 1/2$, we obtain a precise small-ball estimate, from which Slicing is known to follow. In particular, this gives a new and transparent proof that Thin-Shell implies Slicing (by passing through small-ball estimates), recovering (\ref{eq:intro-Lnsigman}). 

\medskip

More generally, in the anisotropic case, we obtain:

\begin{theorem} \label{thm:intro-main2}
Let $\mu$ be a log-concave probability measure on $\R^n$, let $X$ be a random vector distributed according to $\mu$, and denote $\M = \Cov(\mu)$. Then for every $\eps\in(0,1)$,
\[  \mu \!\left(
 |X - \E X|\le \varepsilon\sqrt{\trace \M}
 \right)
 \le
 \exp\!\left\{
 -\frac{(\trace \M)^2}{8 \,\trace(\M^2)}
 \left(1-\varepsilon^2+\varepsilon^2\log\varepsilon^2\right)
 \right\}.
\] \end{theorem}
\noindent Setting $\E X = 0$ and $\M = \Id$ precisely recovers Theorem \ref{thm:intro-main}. Note that $\E |X - \E X|^2 = \trace \M$. 

\subsection{Proof Idea}

The idea in the isotropic case is as follows. For $t\ge 0$ define the Gaussian-tilted measure $\mu_t$ and corresponding partition function $Z(t)$,
\[
    d\mu_t(x) :=\frac{1}{Z(t)}e^{-\frac12t|x|^2}\,d\mu(x),
    \qquad
    Z(t) :=\int_{\R^n}e^{-\frac12t|x|^2}\,d\mu(x),
\]
Let $X_t$ be distributed according to $\mu_t$, and set
\[
    B_t := \E X_t \otimes X_t, \qquad A_t :=\Cov(\mu_t),
    \qquad
    g(t) :=\frac1n\log Z(t).
\]
Note that $A_t = B_t$ when $\mu$ (and hence $\mu_t$) is even, but $A_t = B_t - \E X_t \otimes \E X_t  \preceq B_t$ in general. 
We will use these notations throughout this work. 

We show in Section \ref{sec:2} that the Slicing problem is equivalent to having either (equivalently, both) of the following lower-bounds at some fixed time:
\[
\exists t_0 > 0 \;\;\; \frac{1}{n} \tr B_{t_0} \geq \exp(-C t_0)  \;\; \text{ or } \;\; \frac{1}{n} \log \det A_{t_0} \geq -C t_0 . 
\]
On the other hand, we show in Section \ref{sec:3} that the Thin-Shell conjecture, $\bar \sigma_n^2 \leq 2 D$, is equivalent to having the same control for all times:
\[
\forall t > 0 \;\;\; \frac{1}{n} \tr B_{t} \geq \exp(-D t)   \;\; \text{ or } \;\; \frac{1}{n} \log \det A_{t} \geq -D t  . 
\]
Consequently, Thin-Shell trivially implies Slicing. Since $(\log Z)'(t) =-\frac{1}{2} \tr(B_t)$, integrating the trace lower-bound yields an exponential upper-bound on the partition function $Z(t)$, which translates via Markov's inequality to a precise small-ball estimate -- see Section \ref{sec:4}. In the anisotropic case, to obtain small-ball estimates when the Euclidean norm does not coincide with the covariance structure of $\mu$, we use anisotropic Gaussian tilts, appropriately adapting the argument sketched above -- see Section \ref{sec:5}. Remarkably, the derivation of the small-ball estimate is entirely self-contained, and in fact does not rely on log-concavity at all (see Remark \ref{rem:no-LC}).

\bigskip
\noindent \textbf{AI declaration:} All ideas and results in this work were obtained by the authors. Chat-GPT-5.6-Sol was used to write a draft of various parts of the paper, which was checked and polished by the authors, as well as to prune the paper for typos and consistency. 

\bigskip
\noindent \textbf{Acknowledgment:} We thank the referee for helpful remarks which helped improve our exposition.

\section{Equivalent formulations of Slicing using Gaussian tilts}  \label{sec:2}

We proceed with the notation introduced above. 
Consider the following 6 quantitative properties of $\mu$: 

\begin{enumerate}[label=\textup{(\arabic*)},leftmargin=2.6em]
    \item\label{it:1} \textbf{Slicing Bound.}  For some $C_L > 0$,
    \[
        L_\mu\le C_L.
    \]

    \item\label{it:2} \textbf{Small ball for all radii.}  For some $\varepsilon_0 , C_{\rm sb}>0$, 
     \[
        \mu\bigl(|X|\le \varepsilon\sqrt n\bigr)
        \le (C_{\rm sb}\varepsilon)^n \;\;\; \forall \varepsilon \in (0,\varepsilon_0) .
    \]
 
\item\label{it:3} \textbf{Small ball at one fixed radius.}  For some $\varepsilon_*\in(0,1)$ and $C_{\rm isb}>0$ with
   $C_{\rm isb}\varepsilon_*<1$, 
    \[
        \mu\bigl(|X|\le \varepsilon_*\sqrt n\bigr)
        \le (C_{\rm isb}\varepsilon_*)^n .    \]

 \item\label{it:4} \textbf{Partition function upper bound.}  For some $c_{Z},t_0>0$,
    \[
        g(t_0)=\frac1n\log Z(t_0)\le -c_{Z} t_0 .     \]

    \item\label{it:5} \textbf{Second-moment trace lower bound.}  For some $c_T,t_1>0$,
    \[
        \frac1n \tr(B_{t_1})\ge c_T  . 
    \]

    \item\label{it:6} \textbf{Covariance determinant lower bound.}  For some $C_D,t_2>0$,
    \[
        \frac1n\log\det A_{t_2} \ge -C_D . 
    \]
\end{enumerate}

\begin{proposition}\label{prop:equiv}
Fix an isotropic log-concave measure $\mu$ on $\R^n$. Then:
\begin{align*}
\ref{it:1} & \Leftrightarrow \ref{it:2} \\
& \Downarrow \\
& \ref{it:6} \\
& \Downarrow \\
\ref{it:3} \Leftrightarrow & \ref{it:4} \Leftrightarrow \ref{it:5} 
\end{align*}
If any one of these statements holds uniformly for all isotropic log-concave measures on $\R^n$, then so do all of the other statements. In all of these implications, the parameters $C_L, C_{\rm sb},C_{\rm isb}, \varepsilon_0, \varepsilon_*, c_{Z}, c_T,C_D,t_0,t_1,t_2$ depend only on each other and not on $\mu$ nor on $n$. 
\end{proposition}

The equivalence between \ref{it:1} and \ref{it:2} for a fixed $\mu$ is elementary.  The nontrivial global connection with variants of \ref{it:3} 
was developed by Dafnis and Paouris \cite{DafnisPaouris2010}; see also Bizeul \cite{Bizeul2025} and Brazitikos \cite{Brazitikos2026}.  Here we add the Gaussian-tilt formulations \ref{it:4}--\ref{it:6}. The implication $\ref{it:5}\Rightarrow\ref{it:3}$ was established by 
Brazitikos in \cite[Corollary~4.1]{Brazitikos2026} using a different argument than the one below, involving projection onto subspaces; our simpler argument passes through formulation \ref{it:4}, revealing along the way that \ref{it:3}, \ref{it:4} and \ref{it:5} are in fact equivalent. 

It is useful to record the elementary uniform estimate
\begin{equation} \label{eq:ball}
   0 < c \leq \sqrt n\,|B_2^n|^{1/n} = \sqrt{2\pi e}\,(1+o(1)) \leq C . 
\end{equation}

\begin{proof}[Proof of Proposition \ref{prop:equiv}] 

\noindent\textbf{$\ref{it:1}\Rightarrow\ref{it:2}$.}
For every $\varepsilon>0$,
\[
    \mu\bigl(|X|\le\varepsilon\sqrt n\bigr)
   \le \|f\|_\infty\,\bigl|\varepsilon\sqrt n B_2^n\bigr| 
    =L_\mu^n\,\varepsilon^n n^{n/2}|B_2^n|.
\]
Recalling \eqref{eq:ball}, we see that $L_\mu\le C_L$ implies \ref{it:2} with $C_{\rm sb}=C C_L$. 

\medskip

\noindent\textbf{$\ref{it:2}\Rightarrow\ref{it:1}$.}
Since $f$ is log-concave with barycenter at the origin we have by \cite[Theorem 4]{Fradelizi1997}
\[
    \|f\|_\infty \le e^n f(0) = e^n \lim_{\varepsilon\downarrow0} 
    \frac{\mu(\varepsilon\sqrt n B_2^n)}
         {|\varepsilon\sqrt n B_2^n|} .
\]
Using \ref{it:2}, dividing by $\varepsilon^n$, and letting $\varepsilon\downarrow0$, we obtain
\[
    \|f\|_\infty\,n^{n/2}|B_2^n|
    \le (e \, C_{\rm sb})^n.
\]
Since $\mu$ is isotropic, $L_\mu=\|f\|_\infty^{1/n}$, and therefore
\[     L_\mu
    \le \frac{e \, C_{\rm sb}}{\sqrt n\,|B_2^n|^{1/n}}
    \le \frac{e \, C_{\rm sb}}{c} ,
\] where the last inequality follows from \eqref{eq:ball}. This establishes \ref{it:1}.

\medskip

Of course $\ref{it:2}\Rightarrow\ref{it:3}$ by taking $C_{\rm isb}:=C_{\rm sb}$ and fixing a sufficiently small $\varepsilon_*  \in (0, \min(1,\varepsilon_0))$ such that $C_{\rm isb}\varepsilon_*<1$.
We do not know how to prove the converse implication $\ref{it:3}\Rightarrow\ref{it:2}$ for a fixed $\mu$ -- see Remark \ref{rem:CR} below.  
However, it was shown by
Brazitikos \cite[Corollary~4.1 and the discussion following it]{Brazitikos2026}  that the uniform validity of \ref{it:3} for all isotropic log-concave $\mu$ on $\R^n$ implies back \ref{it:1} -- see Remark \ref{rem:explicit-dependence}. When $\varepsilon_* > 0$ is small-enough, a particularly short proof was also given by Bizeul \cite[Section~3.2]{Bizeul2025}. 
Consequently, if \ref{it:3} holds uniformly over 
all measures under consideration, then so does \ref{it:1}, and hence
\ref{it:2}.

\bigskip

Before proceeding with the proof of the other implications, we first record several elementary facts about the Gaussian tilt.  If $X_t \sim \mu_t$, differentiation under the integral sign gives
\begin{equation}\label{eq:gprime}
    g'(t)=-\frac{1}{2n}\E |X_t|^2
          =-\frac{1}{2n}\tr(B_t),
\end{equation}
and
\begin{equation}\label{eq:gsecond}
    g''(t)=\frac{1}{4n}\Var (|X_t|^2)> 0 .
\end{equation}
Thus $g$ is decreasing and convex, $g(0)=0$, and $\tr(B_t)$ is decreasing. Moreover, Jensen's inequality and isotropicity give
\begin{equation}\label{eq:jensen}
    g(t)=\frac1n\log\E  e^{-t|X|^2/2}
    \ge -\frac{t}{2n}\E |X|^2=-\frac t2.
\end{equation}

\medskip
\noindent\textbf{$\ref{it:1}\Rightarrow\ref{it:6}$.}
Let
\[
    f_t(x)=\frac1{Z(t)}e^{-t|x|^2/2}f(x)
\]
be the density of $\mu_t$.  Clearly \begin{equation}\label{eq:supft}
    \|f_t\|_\infty \le \frac{\|f\|_\infty}{Z(t)}.
\end{equation}
For any probability density $p$ on $\R^n$ with covariance matrix $\Sigma$, the Gaussian maximal-entropy inequality (see e.g.~\cite[Chapter~8]{CoverThomas}) gives
\[
    h(p) := -\int p(x) \log p(x) dx \le \frac12\log\bigl((2\pi e)^n\det\Sigma\bigr),
\]
whereas the elementary bound $p\le\|p\|_\infty$ gives
$h(p)\ge-\log\|p\|_\infty$.  Applying these inequalities to $f_t$ yields
\[
    (\det A_t)^{1/(2n)}
    \ge \frac{1}{\sqrt{2\pi e}}\,\|f_t\|_\infty^{-1/n}.
\]
Using \eqref{eq:supft}, isotropicity, and $L_\mu=\|f\|_\infty^{1/n}$,
\begin{equation}\label{eq:detentropy}
    (\det A_t)^{1/(2n)}
    \ge \frac{Z(t)^{1/n}}{\sqrt{2\pi e}\,L_\mu}.
\end{equation}
By \eqref{eq:jensen},
\[
    Z(t)^{1/n}\ge e^{-t/2},
\]
and so if $L_\mu\le C_L$,
\[
    \frac1n\log\det A_t
    \ge -\log(2\pi e)-2\log C_L-t.
\]
Setting $t = t_2$ for any fixed $t_2 > 0$ establishes \ref{it:6}.

\medskip
\noindent\textbf{$\ref{it:6}\Rightarrow\ref{it:5}$.}
By the arithmetic--geometric mean inequality applied to the eigenvalues of $A_{t_2}$,
\[
    \frac1n\tr(B_{t_2}) \ge \frac1n\tr(A_{t_2})
    \ge (\det A_{t_2})^{1/n}
    \ge e^{-C_D}.
\]
Thus \ref{it:5} holds with $t_1=t_2$ and $c_T=e^{-C_D}$.

\bigskip

\medskip
\noindent\textbf{$\ref{it:3}\Rightarrow\ref{it:4}$.}
Let $q = C_{\rm isb}\varepsilon_* <1$, so that
\[
    \mu\bigl(|X|\le\varepsilon_*\sqrt n\bigr)\le q^n.
\]
For every $t>0$, splitting the integral defining $Z(t)$ at the sphere of radius $\varepsilon_* \sqrt n$ gives
\begin{equation}\label{eq:splitZ}
    Z(t)
    \le \mu\bigl(|X|\le\varepsilon_*\sqrt n\bigr)
      +e^{-t\varepsilon_*^2 n/2} 
    \le q^n+e^{-t\varepsilon_*^2 n/2}.
\end{equation}
Choose
\begin{equation}\label{eq:Tchoice}
    t= t_0 := \frac{2}{\varepsilon_*^2}\log\frac{2}{1-q} >  0. 
\end{equation}
Then $e^{-t_0 \varepsilon_*^2/2}=\frac{1-q}{2}$, and therefore, using $a^n+b^n\le(a+b)^n$ for $a,b\ge0$,
\[
    Z(t_0)\le
    \left(q+\frac{1-q}{2}\right)^n
    =\left(\frac{1+q}{2}\right)^n.
\]
Set
\[
   c_{Z}:=\frac{1}{t_0} \log\frac{2}{1+q} >0.
\]
Then $g(t_0)= \frac{1}{n} \log Z(t_0) \le-c_{Z}t_0$, establishing \ref{it:4}.

\medskip
\noindent\textbf{$\ref{it:4}\Rightarrow\ref{it:5}$.}
Convexity of $g$ implies that $g(t)/t$ is non-decreasing, and hence \textbf{\ref{it:4}} is actually equivalent to the statement that
\[
g(t) \le \frac{t}{t_0} g(t_0) \leq -c_{Z} t  \;\;\; \forall t \in (0,t_0) . 
\]
On the other hand, by \eqref{eq:jensen},
\[
    -\frac t2\le g(t) ; 
\]
in particular 
\begin{equation} \label{eq:ciz}
c_{Z}\le 1/2.
\end{equation}
  If $0<t<s<t_0$, convexity of $g$ gives
\[
    g'(t)\le\frac{g(s)-g(t)}{s-t}
    \le\frac{-c_{Z}s+t/2}{s-t}.
\]
Letting $s\uparrow t_0$ yields
\begin{equation}\label{eq:derivbound}
    g'(t)\le
    \frac{-c_{Z}t_0+t/2}{t_0-t}.
\end{equation}
Using \eqref{eq:gprime},
\begin{equation}\label{eq:tracefromg}
    \frac1n\tr(B_t)
    =-2g'(t)
    \ge \frac{2c_{Z}t_0-t}{t_0-t}.
\end{equation}
For instance, for $0<t\le c_{Z}t_0$ the right-hand side is bounded below by
\[
    \frac{c_{Z}}{1-c_{Z}}>0.
\]
Thus \ref{it:5} holds with, say,
\[
    t_1=c_{Z}t_0,
    \qquad
    c_T=\frac{c_{Z}}{1-c_{Z}}.
\]

\medskip
\noindent\textbf{$\ref{it:5}\Rightarrow\ref{it:4}$.}
Assume that
\[
    \frac1n\tr(B_{t_1})\ge c_T .
\]
Recall that $g'(t) = -\frac{1}{2n}\tr(B_t)$ is increasing and that $g(0) = 0$. Therefore, 
\[
    g(t_1) =\int_0^{t_1}g'(t)\,dt =-\frac12\int_0^{t_1}\frac1n\tr(B_t)\,dt \le -\frac12\int_0^{t_1}\frac1n\tr(B_{t_1})\,dt \leq  -\frac{c_T}{2}t_1 .
\]
Thus \ref{it:4} holds with
\[
    t_0=t_1,
    \qquad
    c_{Z}=\frac{c_T}{2}.
\]

\medskip
\noindent\textbf{$\ref{it:4}\Rightarrow\ref{it:3}$.}
Assume that
\[
    Z(t_0) = e^{n g(t_0)} \le e^{-c_{Z}t_0n}.
\]
For every $\varepsilon>0$, on the event
$\{|X|\le\varepsilon\sqrt n\}$ we have
\[
    e^{-t_0|X|^2/2}
    \ge e^{-t_0\varepsilon^2n/2}.
\]
Consequently
\begin{equation}\label{eq:smallball-from-Z}
    \mu\bigl(|X|\le\varepsilon\sqrt n\bigr)
    \le e^{t_0\varepsilon^2n/2} Z(t_0) \leq \exp\left\{
        -t_0\left(c_{Z}-\frac{\varepsilon^2}{2}\right)n
    \right\}.
\end{equation}
Choose any 
\[
    \varepsilon_* \in (0, \sqrt{2 c_{Z}}) ;
\]
note that by (\ref{eq:ciz}), $\sqrt{2 c_{Z}} \leq 1$, and so $\varepsilon_* < 1$. 
Then
\[
    q:=
    \exp\left\{
        -t_0\left(c_{Z}-\frac{\varepsilon_*^2}{2}\right)
    \right\}
    <1,
\]
and so defining $C_{\rm isb}:=\frac{q}{\varepsilon_*}$, \eqref{eq:smallball-from-Z} gives
\[
    \mu\bigl(|X|\le\varepsilon_*\sqrt n\bigr)
    \le (C_{\rm isb}\varepsilon_*)^n.
\]
This proves \ref{it:3}.
\end{proof}

\begin{remark} \label{rem:CR}
Starting from the single-radius small-ball estimate \ref{it:3} for a fixed \textbf{even} log-concave $\mu$, it is possible to obtain the following small-ball estimate for all smaller radii:
     \[
        \mu\bigl(|X|\le \varepsilon\sqrt n\bigr)
        \le (C_{\rm sb}\varepsilon)^{\beta n }\;\;\; \forall \varepsilon \in (0,\varepsilon_*) ,
    \]
    with some $C_{\rm sb},\beta > 0$ depending on $C_{\rm isb},\varepsilon_*$. Indeed, it follows from \cite[Corollary 2]{CorderoRotem2023} that $\R \ni s\mapsto \mu(|X| \leq e^s)$ is log-concave, and so an upper bound on $\mu( |X| \leq \varepsilon \sqrt{n} )$ is immediately deduced from an upper bound for $\varepsilon = \varepsilon_*$ and a lower-bound for $\varepsilon = 2$ (obtained by a simple Markov estimate). 
            However, this gives the wrong exponent of $\beta n$ instead of $n$ as in \ref{it:2}. 
\end{remark}

\begin{remark} \label{rem:explicit-dependence}
For future reference, we note that if \ref{it:3} holds for all isotropic log-concave measures $\mu$ on $\R^n$, then \ref{it:1} holds with
\[
C_L \leq \frac{C}{\varepsilon_* \sqrt{1 - C_{\rm isb} \varepsilon_*}} ,
\]
for some universal $C > 0$. This follows by carefully inspecting the proof by Brazitikos in \cite{Brazitikos2026}.
\end{remark}

\section{Equivalent formulations of Thin-Shell using Gaussian tilts} \label{sec:3}

From here on log-concavity will not be used at all, we just need to assume that the quantities below are finite. So fix an isotropic probability measure $\mu$ on $\R^n$ with, say, finite fourth moments. 
For $t>0$, define
\[
    b_t := \E X_t \qquad m(t):=\tr(B_t) = \E |X_t|^2,
    \qquad
    \Phi(t):=\log\det A_t ,
\]
where recall $X_t\sim\mu_t$, $B_t = \E X_t \otimes X_t$ and $A_t = \Cov(\mu_t) = B_t - b_t \otimes b_t$. Differentiating the barycenter and covariance along the Gaussian tilt gives
\begin{align*}
b_t' &= -\frac12\,\Cov\!\left ( X_t , |X_t|^2 \right ) , \\
 B_t' &=  -\frac12\,\Cov\!\left(X_t \otimes X_t,|X_t|^2\right)  , \\
 A_t' & =-\frac12\,\Cov\!\left((X_t - b_t)\otimes (X_t-b_t),|X_t|^2\right)  .
\end{align*}
 Taking traces, we have
\begin{equation}\label{eq:mprime}
    m'(t)=\tr(B_t') = -\frac12\Var (|X_t|^2)\le0,
    \qquad m(0)=n ,
\end{equation}
and moreover,
\begin{equation}
    \Phi'(t)
    =\tr(A_t^{-1}A_t') =-\frac12\Cov \!\left(
        \langle A_t^{-1} (X_t - b_t), (X_t - b_t) \rangle,|X_t|^2
    \right).
    \label{eq:phiprime}
\end{equation}
Define the random vector
\[
    Y_t := A_t^{-1/2} (X_t - b_t) \sim \nu_t  ,
\]
so that $\nu_t$ is again an isotropic probability measure on $\R^n$. Note that if $\mu$ was log-concave then all $\nu_t$'s will remain log-concave as well. We see that
\begin{equation}\label{eq:isotropized-var}
    \Var \!\left(\langle A_t^{-1} (X_t-b_t),(X_t-b_t)\rangle\right)
    =\Var (|Y_t|^2).
\end{equation}

We are now ready to prove the following proposition. Its usefulness stems from the fact that $\nu_t$ remains isotropic for all $t > 0$. In some sense, this may be thought of as a deterministic version of Eldan's stochastic localization method \cite{Eldan2013}, in which $\mu$ is disintegrated into stochastically evolving non-centered Gaussian tilts, which are constantly put in isotropic position. In contrast, we just need to control the deterministic central Gaussian tilt. 

\begin{proposition}\label{prop:thin-shell-flow}
Let $\mu$ be a fixed isotropic probability measure on $\R^n$ with finite fourth moments, and let $D>0$ and $t_0>0$.  Then each of the following statements implies the one after it:
\begin{enumerate}[label=\textup{(\Alph*)},leftmargin=2.6em]
    \item\label{it:tsA}
    \[
       \frac{1}{n} \Var_{Y_t \sim \nu_t} (|Y_t|^2)\le 2D 
        \qquad \forall\,t\in(0,t_0).
    \]

    \item\label{it:tsB}
    \[
        \frac1n\log\det A_t\ge -Dt
        \qquad \forall\,t\in(0,t_0).
    \]

    \item\label{it:tsC}
    \[
        \frac1n\tr(B_t)\ge e^{-Dt}
        \qquad \forall\,t\in(0,t_0).
    \]

    \item\label{it:tsD}
    \[
       \frac{1}{n} \Var_{X \sim \mu} (|X|^2)\le 2D .
    \]
\end{enumerate}
\end{proposition}

\begin{proof}

Assume \ref{it:tsA}.  By Cauchy--Schwarz, \eqref{eq:mprime}, \eqref{eq:phiprime}, and \eqref{eq:isotropized-var},
\begin{equation}
    |\Phi'(t)|
    \le \frac12
    \sqrt{\Var (|Y_t|^2)\,
           \Var (|X_t|^2)} 
    \le \sqrt{Dn\,(-m'(t))}.
    \label{eq:phi-prime-bound}
\end{equation}
Since $m$ is non-increasing and $m(0)=n$, the arithmetic--geometric mean inequality gives $\det(A_t)^{1/n}\le \tr(A_t)/ n \leq  m(t)/n\le1$, and hence $\Phi(t)\le0$. Using Cauchy--Schwarz in the time variable, we obtain
\begin{equation}
    -\Phi(t)
    \le \int_0^t  -\Phi'(s)\,ds \le     \sqrt{Dn\,t\int_0^t(-m'(s))\,ds} 
    =\sqrt{Dn\,t\,(n-m(t))}.
    \label{eq:phi-integrated}
\end{equation}
On the other hand, applying $-\log x\ge 1-x$ to the eigenvalues of $A_t$ and summing yields
\begin{equation}\label{eq:phi-lower-trace}
    -\Phi(t)\ge n- \tr A_t \geq n - m(t).
\end{equation}
Writing $h(t)=n-m(t)\ge0$, \eqref{eq:phi-integrated} and \eqref{eq:phi-lower-trace} imply
\[
    h(t)\le \sqrt{Dn\,t\,h(t)},
\]
and hence
\begin{equation}\label{eq:trace-linear-D}
    n-m(t)\le Dnt.
\end{equation}
Substituting this back into \eqref{eq:phi-integrated} gives
\[
    -\Phi(t)\le Dnt,
\]
which is precisely \ref{it:tsB}.

Next, \ref{it:tsB} implies \ref{it:tsC} by the arithmetic--geometric mean inequality:
\[
    \frac1n\tr(B_t) \ge \frac1n\tr(A_t)
    \ge (\det A_t)^{1/n}
    \ge e^{-Dt}.
\]

Finally, assume \ref{it:tsC}.  Since $m(0)=n$,
\[
    \frac{m(t)-m(0)}{t}
    \ge n\,\frac{e^{-Dt}-1}{t}.
\]
Letting $t\downarrow0$ and using \eqref{eq:mprime} at $t=0$ gives
\[
    -\frac12\Var(|X|^2)=m'(0)\ge -Dn,
\]
which is \ref{it:tsD}.
\end{proof}

\begin{corollary}\label{cor:uniform-thin-shell-flow}
Fix $D > 0$. The following statements are equivalent:
\begin{enumerate}[label=\textup{(\Alph*)},leftmargin=2.6em]
\item \label{it:tscA} \textbf{Uniform Thin-Shell.} For all isotropic log-concave probability measures $\mu$ on $\R^n$,
\[
    \frac{1}{n} \Var_{X \sim \mu} (|X|^2)\le 2 D  .
\]  
\item \label{it:tscB} \textbf{Uniform covariance determinant control.}
For all isotropic log-concave probability measures $\mu$ on $\R^n$,
\[
    \frac1n\log\det A_t\ge -Dt \;\;\;\; \forall t >  0 . 
\]
\item \label{it:tscC} \textbf{Uniform second-moment trace control.}
For all isotropic log-concave probability measures  $\mu$ on $\R^n$,
\[
    \frac1n\tr(B_t)\ge e^{-Dt} \;\;\;\; \forall t >  0 . 
\]
\end{enumerate}
\end{corollary}
\begin{proof}
Statements \ref{it:tscB} or \ref{it:tscC} directly imply \ref{it:tscA} by Proposition~\ref{prop:thin-shell-flow} \ref{it:tsD}. Conversely, assume \ref{it:tscA}, and let $\mu$ be an isotropic log-concave measure on $\R^n$. 
For every $t>0$,  $Y_t := A_t^{-1/2} (X_t - b_t)$ is distributed according to the isotropic log-concave measure $\nu_t$. Applying \ref{it:tscA} to $\nu_t$ verifies
\[
    \frac{1}{n} \Var_{Y_t \sim \nu_t} (|Y_t|^2)\le 2D
    \qquad \forall\,t>0.
\]
Statements \ref{it:tscB} and \ref{it:tscC} now follow immediately by Proposition~\ref{prop:thin-shell-flow}. 
\end{proof}

\section{Thin-Shell implies Slicing} \label{sec:4}

Combining the results of the previous two sections, we can now immediately verify that Thin-Shell implies Slicing. This was first observed by Eldan and Klartag \cite{EldanKlartag} using a very different argument. 

Recall from the Introduction that
\[
L_n := \sup  L_\mu ~,~ \bar \sigma_n^2 :=   \sup \frac{1}{n} \Var_{X \sim \mu}(|X|^2) ,
\]
where in both cases the supremum is over all isotropic log-concave measures $\mu$ on $\R^n$. By testing for example the standard Gaussian measure, we have $L_n \geq \frac{1}{\sqrt{2 \pi}}$ and $\bar \sigma_n^2 \geq 2$. Simple estimates show that $L_n$ and $\bar \sigma_n$ are both finite for each $n \geq 1$. 
We say that uniform Slicing (Thin-Shell) holds if $L_n \leq C$ ($\bar \sigma_n^2 \leq C$) for some universal constant $C>0$ and all $n \geq 1$.

\begin{corollary}\label{cor:thin-shell-slicing} 
Uniform Thin-Shell implies uniform Slicing. Moreover, for every $n \geq 1$,
\begin{equation} \label{eq:Lnsigman}
L_n \leq C \bar \sigma_n 
\end{equation}
for some universal constant $C > 0$. 
\end{corollary}
\begin{proof}
By definition, statement \ref{it:tscA} of  Corollary~\ref{cor:uniform-thin-shell-flow} holds with $D = \bar \sigma_n^2 /2$. 
It follows by statement \ref{it:tscB} of Corollary~\ref{cor:uniform-thin-shell-flow} that for any $t > 0$ ,
\[
    \frac1n\log\det A_t\ge -D t ,
\]
which is statement \ref{it:6} of Proposition~\ref{prop:equiv}.  Alternatively, by statement \ref{it:tscC} of Corollary~\ref{cor:uniform-thin-shell-flow}, we also have for any $t > 0$,
\[
    \frac1n\tr(B_t)\ge e^{-D t},
\]
which is statement \ref{it:5} of Proposition~\ref{prop:equiv}.  Proposition~\ref{prop:equiv} therefore immediately shows that a universal upper bound on $\bar\sigma_n$ (and hence $D$) implies a universal upper bound on $L_n$. 

To get the linear dependence in (\ref{eq:Lnsigman}), let us choose $t = t_D = \frac{1}{D}$ above, so that $c_T = \frac{1}{e}$ in statement \ref{it:5} of Proposition~\ref{prop:equiv}.
Inspecting the quantitative relation between parameters used in the implications $\ref{it:5} \Rightarrow \ref{it:4} \Rightarrow \ref{it:3}$, we see that we may choose 
\[
c_{Z} = \frac{1}{2 e} ~,~ \varepsilon_* = \frac{1}{2} ~,~ q := C_{isb} \varepsilon_* = \exp(-t_D (c_{Z} - 1/8)) \leq \exp(-c /D).
\]
 It follows from Remark \ref{rem:explicit-dependence} (and $D \geq 1$) that 
 \[
 L_n \leq \frac{C}{\sqrt{1 - q}} \leq C' \sqrt{D} \leq C'' \bar \sigma_n,
 \]
  confirming (\ref{eq:Lnsigman}). 
\end{proof}

In fact, if we do not simply use $t = \frac{1}{D}$ above, but rather utilize the decay of $\tr(B_s)$ for all $s \in [0,t]$ and optimize over the choice of $t > 0$, we obtain the following very precise new small-ball estimate (note that there are no hidden universal constants in the formulation):

\begin{theorem}\label{thm:trace-small-ball}
Assume the following uniform thin-shell bound in $\R^n$:
\[
\bar \sigma^2_n \leq 2 D.
\] 
Then, for every $\eps \in(0,1)$,
\begin{equation}
\label{eq:lower-tail-from-trace}
    \mu\left(|X|\le \eps \sqrt n\right)
    \le
    \exp\left\{
        -\frac{n}{2D}
        \brac{
            1-\eps^2
            +\eps^2\log \eps^2
        }
    \right\} \\
\end{equation}
or equivalently, setting $s = (1-\eps) \sqrt{n}$, for all $s \in (0 , \sqrt{n})$,
\[
    \mu\left(|X|\le \sqrt{n} - s \right ) \le \exp \left \{  - \frac{s^2}{D} \left (1 + O\brac{\frac{s}{\sqrt{n}}} \right ) \right \} . 
\]
\end{theorem}
\begin{proof}
Recall that $g(t):=\frac1n\log Z(t)$
satisfies
\[
    g'(t)=-\frac1{2n}\tr(B_t),
    \qquad
    g(0)=0.
\]
Hence, by Corollary \ref{cor:uniform-thin-shell-flow}, 
\[
    g(t) =-\frac12\int_0^t \frac1n\tr(B_s)\,ds \le -\frac12\int_0^t e^{-Ds}\,ds =-\frac{1-e^{-Dt}}{2D}.
\]
Consequently,
\[     Z(t)\le
    \exp\left\{
        -\frac{n}{2D}(1-e^{-Dt})
    \right\}.
\] 
Fix $\varepsilon\in(0,1)$. As usual, on the event
$\{|X|\le\varepsilon\sqrt n\}$, we have
\[
    e^{-t|X|^2/2}\ge e^{-t\varepsilon^2n/2},
\]
and therefore for all $t > 0$,
\begin{equation}\label{eq:small-ball-before-opt}
    \mu\bigl(|X|\le\varepsilon\sqrt n\bigr)
    \le e^{t\varepsilon^2n/2}Z(t)
    \le
    \exp\left\{
        \frac n2
        \left[
            t\varepsilon^2-\frac{1-e^{-Dt}}{D}
        \right]
    \right\}.
\end{equation}
It remains to optimize in $t>0$. Set
\[
    \phi(t)
    :=
    t\varepsilon^2-\frac{1-e^{-Dt}}{D}.
\]
Then
\[
    \phi'(t)=\varepsilon^2-e^{-Dt},
    \qquad
    \phi''(t)=De^{-Dt}>0,
\]
so the unique minimizer is given by
\[
    t_*=\frac1D\log\frac1{\varepsilon^2} > 0.
\]
Substituting $t_*$ into \eqref{eq:small-ball-before-opt} yields \eqref{eq:lower-tail-from-trace}. 

Finally, writing $\eps = 1 - \delta$ and setting
\[
    \Psi(\delta)
    :=
    1-(1-\delta)^2
    +(1-\delta)^2\log(1-\delta)^2,
\]
then
\[
    \Psi(\delta)
    =
    2\delta^2-\frac23\delta^3+O(\delta^4)
    \qquad (\delta\to0),
\]
which gives the stated asymptotic form of the exponent when $\delta = \frac{s}{\sqrt{n}}$. \end{proof}

Using the Chen--Klartag estimate $\bar \sigma_n^2 \leq 8$, Theorem \ref{thm:intro-main} immediately follows. 

\begin{remark} \label{rem:no-LC}
Note that log-concavity was actually never used in the proofs of Proposition \ref{prop:thin-shell-flow}, Corollary \ref{cor:uniform-thin-shell-flow} nor Theorem \ref{thm:trace-small-ball} -- all we need to obtain deviation estimates for $|X|$ below the expectation is a uniform control over the thin-shell variance parameter $\frac{1}{n} \Var_{X \sim \mu}(|X|^2)$ for a class of isotropic probability measures $\mu$ which is closed under applying a central Gaussian tilt and taking the isotropic affine image. 
\end{remark}

\section{The anisotropic case} \label{sec:5}

The above argument also extends to the anisotropic case under the natural assumption (\ref{eq:ani-var}) below. In order to obtain an exact analogue of Theorem \ref{thm:trace-small-ball} without incurring any loss, the computations turn out to be somewhat more intricate. 

\begin{theorem}\label{thm:anisotropic}
Assume that for every isotropic log-concave probability measure $\mu$ on $\R^n$ and symmetric operator $M : \R^n \rightarrow \R^n$,
\begin{equation}\label{eq:ani-var}
 \Var_{Y \sim \mu}(\scalar{MY,Y}) \le 2 D \tr(M^2) . 
\end{equation} 
Let $\nu$ be a  log-concave probability measure on $\R^n$, let $Z \sim \nu$, and denote $\M := \Cov(\nu)$. 
Then, for every $\varepsilon\in(0,1)$, 
\begin{equation}\label{eq:main}
 \nu \!\left(
 |Z - \E Z|\le \varepsilon\sqrt{\trace \M}
 \right)
 \le
 \exp\!\left\{
 -\frac{(\trace \M)^2}{2 D \,\trace(\M^2)}
 \left(1-\varepsilon^2+\varepsilon^2\log\varepsilon^2\right)
 \right\}.
\end{equation}
\end{theorem}

It was observed by Letwin \cite{Letwin2026} (see also \cite{Bizeul2026}) that the Chen--Klartag argument which yields the sharp variance bound on $|Y|^2$ for general isotropic log-concave probability measures, equally applies to general quadratic functions $\scalar{M Y,Y}$ in the form (\ref{eq:ani-var}), with the same constant $D = 4$ as for $M = \Id$. Consequently, Theorem \ref{thm:anisotropic} immediately yields Theorem \ref{thm:intro-main2}. 

\medskip

For the proof of Theorem \ref{thm:anisotropic}, we will need the following formula for differentiation of the logarithm of a positive-definite matrix (see e.g.~\cite[Theorem V.3.3]{Bhatia1997} for a general treatment):

\begin{lemma} \label{lem:log}
Let $A$ denote a positive-definite $n \times n$ matrix, and choose an orthonormal basis in which
$$
A=\diag(\lambda_1,\ldots,\lambda_n).
$$
Then for any  $n \times n$ symmetric matrix $Q$,
$$
\left(
\left.\frac{d}{ds}\right|_{s=0}\log(A+sQ)
\right)_{ij}
=
Q_{ij}\,\Lambda_{ij},
$$
where
$$
\Lambda_{ij}
:=
\begin{cases}
	\displaystyle
	\frac{\log\lambda_i-\log\lambda_j}{\lambda_i-\lambda_j}
	& \lambda_i\neq\lambda_j,\\[1.2ex]
	\displaystyle
	\frac{1}{\lambda_i}
	& \lambda_i=\lambda_j.
\end{cases}
$$
\end{lemma}
\begin{remark}
Note that the second line is precisely the continuous extension of the first one
to the diagonal $\lambda_i=\lambda_j$.
Also note that this formula is independent of the particular orthonormal
basis chosen inside an eigenspace of $A$, 
so there is no ambiguity in the definition of $\Lambda$. 
\end{remark}

\begin{corollary} \label{cor:log}
For any two $n \times n$ symmetric matrices $P,Q$,
$$
\trace\left(
P\left.\frac{d}{ds}\right|_{s=0}\log(A+sQ)
\right)
=
\trace\left(
Q\left.\frac{d}{ds}\right|_{s=0}\log(A+sP)
\right).
$$
\end{corollary}
\begin{proof}
Since $\Lambda$ is symmetric, in the above eigenbasis both sides are equal to $\sum_{i,j} P_{ij} Q_{ij}\Lambda_{ij}$. 
\end{proof}

\begin{proof}[Proof of Theorem \ref{thm:anisotropic}]
Without loss of generality we may assume that $\E Z = 0$ and that $\M$ is positive-definite, otherwise simply translate $\nu$ and project it onto the image of $\M$. 

Let $\mu := (\M^{-1/2})_* \nu$ denote the distribution of $X := \M^{-1/2} Z$, so that $\mu$ is isotropic and log-concave, and $|Z|^2=\scalar{\M X,X}$. 
Introduce the anisotropic Gaussian tilts for $t \geq 0$:
\[
 X_t \sim \mu_t ~,~ d\mu_t(x) :=\frac{1}{Z(t)}e^{-\frac t2\scalar{\M x,x}}\, d\mu(x) ~,~ Z(t) := \int_{\R^n} e^{-\frac t2 \scalar{\M x,x}}\,d\mu(x) .
 \]
 Denote as usual
 \[
 B_t := \E X_t \otimes X_t ~,~ A_t:=\covariance(\mu_t) ~,~ b_t := \E X_t .
 \]
 The anisotropic analogues of the second-moment trace and covariance determinant are given by
 \[
 m(t):=\trace(\M B_t) ~,~ \Phi(t):=\trace\big(\M\log A_t\big).
 \]
 As usual, differentiation yields
\begin{align} 
\label{eq:ani-Aprime}
 A_t' & = -\frac12\,\Cov \brac{(X_t - b_t)\otimes (X_t-b_t), \scalar{\M X_t,X_t}}  ,\\
\label{eq:ani-mprime}
 B_t' & = -\frac12\,\Cov \brac{X_t\otimes X_t, \scalar{\M X_t,X_t}} ~,~ m'(t)=-\frac12 \Var(\scalar{ \M X_t,X_t}) . 
 \end{align}
As for the derivative of $\Phi(t)$, let
\[
K_t := \left.\frac{d}{ds}\right|_{s=0}\log(A_t+s \M).
\]
It follows by Corollary \ref{cor:log} and \eqref{eq:ani-Aprime} that
\begin{align}
\nonumber \Phi'(t) & = \trace\left( \M\left.\frac{d}{ds}\right|_{s=0} \log(A_t+sA_t') \right) = \trace(A_t' K_t) \\
\nonumber	& =
	-\frac12\,
	\covariance\!\left(
	\langle K_t (X_t-b_t),(X_t-b_t)\rangle,
	\langle \M X_t,X_t\rangle
	\right) \\
\label{eq:ani-Phi-derivative} 	& = -\frac12\,
	\covariance\!\left(
	\langle M_t Y_t,Y_t\rangle,
	\langle \M X_t,X_t\rangle
	\right) ,
\end{align}
where
\[
Y_t :=A_t^{-1/2} (X_t -b_t) ~,~
M_t :=A_t^{1/2}K_t A_t^{1/2}.
\]

Now fix $t \geq 0$ and choose an orthonormal basis in which
\[
A_t=\diag(\lambda_1,\ldots,\lambda_n).
\]
By Lemma \ref{lem:log},
\[
(K_t)_{ij}=\M_{ij}\Lambda_{ij},
\]
and hence
\[
(M_t)_{ij} = \M_{ij}\sqrt{\lambda_i\lambda_j}\,\Lambda_{ij}.
\]
Denote
$$
u_{ij}:=\frac12\log\frac{\lambda_i}{\lambda_j},
$$
and observe that
\[
\sqrt{\lambda_i\lambda_j}\,\Lambda_{ij} =  \begin{cases} \displaystyle \frac{\log(\lambda_i/\lambda_j)}{\sqrt{\lambda_i/\lambda_j} - \sqrt{\lambda_j/\lambda_i}} = \frac{u_{ij}}{\sinh u_{ij}} & \lambda_i \neq \lambda_j \\ & \\ \displaystyle 1 & \lambda_i = \lambda_j \end{cases} .
\]
Since
$$
\left|\frac{u}{\sinh u}\right|\le 1,
$$
we conclude that
\begin{equation}\label{eq:Bt-HS}
	\trace(M_t^2)
	=
	\sum_{i,j}|(M_t)_{ij}|^2
	\le
	\sum_{i,j}|\M_{ij}|^2
	=
	\trace(\M^2) .
\end{equation}
Note that this ingredient was automatically fulfilled in the isotropic case when $\M = \Id$, because $K_t = A_t^{-1}$ and hence $M_t = \Id$. 

We can now finally repeat the proof from the isotropic case. Denote $S_1 :=\tr(\M)$ and $S_2:=\tr(\M^2)$.
Since $Y_t$ is isotropic, applying \eqref{eq:ani-var} and \eqref{eq:Bt-HS} to $Y_t$ and $M_t$ yields
\[
\Var(\scalar{M_t Y_t,Y_t}) \leq 2 D \tr(M_t^2)  \leq 2 D S_2 . 
\]
Recalling \eqref{eq:ani-mprime} and \eqref{eq:ani-Phi-derivative} and applying Cauchy--Schwarz, we obtain:
\begin{equation}\label{eq:ani-phiprime}
	|\Phi'(t)|\le \frac{1}{2} \sqrt{ \Var(\scalar{M_t Y_t,Y_t}) \Var(\scalar{\M X_t,X_t})} 
	\le \sqrt{D S_2 \,(-m'(t))}.
\end{equation}
On the other hand, since $-\log A_t\succeq I-A_t \succeq I - B_t$ in the positive-definite order, and since $P,Q \succeq 0$ implies that $\tr(P Q) \geq 0$,  tracing with $\M$ yields
\[  -\Phi(t)\ge S_1 -m(t).
\] Hence, integrating \eqref{eq:ani-phiprime}, applying Cauchy--Schwarz in the time variable, and using that $A_0 = B_0 = \Id$ and hence $\Phi(0) = 0$ and $m(0) = S_1$, we obtain
$$
S_1-m(t)
 \le -\Phi(t) = \int_0^t -\Phi'(s) ds 
 \le \sqrt{D S_2 t\,(S_1 -m(t))},
$$
and therefore
\begin{equation}\label{eq:h}
 S_1-m(t)\le D S_2 t ~,~ \Phi(t) \geq - D S_2 t . 
\end{equation}
To upgrade the linear lower bound on $m(t)$ to an exponential one via $\Phi(t)$, we use the following weighted arithmetic-geometric mean variant. Indeed, since $\M^{(1)}:=\frac{1}{S_1} \M$ is positive-definite and $\tr(\M^{(1)})=1$, using the diagonalizing basis for $A_t$ and applying Jensen's inequality, we have
\begin{align*}
& \log\frac{m(t)}{S_1} =  \log \tr(\M^{(1)} B_t) \ge
 \log \tr(\M^{(1)} A_t) \\ 
 & = \log \sum_{i=1}^n \M^{(1)}_{ii} \lambda_i  \ge \sum_{i=1}^n \M^{(1)}_{ii} \log \lambda_i = \tr( \M^{(1)} \log A_t)  = \frac{\Phi(t)}{S_1} .
\end{align*}
Together with \eqref{eq:h}, this gives the anisotropic trace exponential lower bound
\begin{equation}\label{eq:tracebound}
 m(t)\ge S_1 \exp\!\left(- D \frac{S_2}{S_1} t\right).
\end{equation}

Finally,
$$
 (\log Z)'(t)=-\frac12m(t),
$$
so integrating \eqref{eq:tracebound} yields
\begin{equation}\label{eq:partition}
 \log Z(t)
 \le
 -\frac{S^2_1}{2 D S_2}
 \left(1-e^{-D \frac{S_2}{S_1} t}\right).
\end{equation}
On the event $\{\scalar{\M X,X} \le \varepsilon^2 S_1\}$,
$$
 e^{-\frac t2 \scalar{\M X,X}} \ge e^{- \frac t2 \varepsilon^2 S_1},
$$
and so by \eqref{eq:partition}, for all 
$u:=\frac{D S_2}{S_1}t > 0$,
$$
 \mu \big(\langle \M X,X\rangle\le \varepsilon^2 S_1 \big)
 \le  e^{\frac t2 \varepsilon^2 S_1} Z(t) \le
 \exp\!\left\{
 \frac{S_1^2}{2 D S_2}\brac{
 \varepsilon^2u-(1-e^{-u})}\right\} . 
$$
Optimizing on $u > 0$, the right-hand side is minimal for $e^{-u}=\varepsilon^2$, i.e.
$u=\log(1/\varepsilon^2)$, yielding
$$
  \mu \big(\langle \M X,X\rangle\le \varepsilon^2 S_1\big)
 \le
 \exp\!\left\{
 -\frac{S_1^2}{2 D S_2}
 \left(1-\varepsilon^2+\varepsilon^2\log\varepsilon^2\right)
 \right\}.
$$
Since $\langle \M X,X\rangle=|Z|^2$, this is exactly \eqref{eq:main}.
\end{proof}

\end{document}